\documentclass[12pt,A4paper]{article}

\usepackage[left=32mm,right=32mm,top=30mm,bottom=32mm]{geometry}
\usepackage{labelfig}
\usepackage{epsfig}
\usepackage{epstopdf}

\usepackage{xfrac}

\usepackage{color}
\usepackage{amsthm,amsmath,amssymb}
\usepackage{booktabs}
\usepackage{txfonts}
\usepackage{mathpazo}
\usepackage{microtype}
\usepackage{overpic}
\usepackage{bm}
\usepackage{sectsty}
\usepackage[
	pdftitle={PDFTitle},
	pdfauthor={Hugo Parlier and Lionel Pournin},
	ocgcolorlinks,
	linkcolor=linkred,
	citecolor=linkred,
	urlcolor=linkblue]
{hyperref}

\definecolor{linkred}{rgb}{0.48,0.1,0.05}
\definecolor{linkblue}{RGB}{16, 78, 139}

\usepackage[hang,flushmargin]{footmisc}
\usepackage{enumitem}
\usepackage{titlesec}
	\titlespacing{\section}{0pt}{12pt}{0pt}
	\titlespacing{\subsection}{0pt}{6pt}{0pt}
	
\titlelabel{\thetitle.\quad}

\makeatletter 

\long\def\@footnotetext#1{%
\H@@footnotetext{%
\ifHy@nesting 
\hyper@@anchor{\@currentHref}{#1}%
\else 
\Hy@raisedlink{\hyper@@anchor{\@currentHref}{\relax}}#1%
\fi 
}}

\def\@footnotemark{%
\leavevmode 
\ifhmode\edef\@x@sf{\the\spacefactor}\nobreak\fi 
\H@refstepcounter{Hfootnote}%
\hyper@makecurrent{Hfootnote}%
\hyper@linkstart{link}{\@currentHref}%
\@makefnmark 
\hyper@linkend 
\ifhmode\spacefactor\@x@sf\fi 
\relax 
}%

\ifFN@multiplefootnote%
\renewcommand*\@footnotemark{%
\leavevmode 
\ifhmode 
\edef\@x@sf{\the\spacefactor}%
\FN@mf@check 
\nobreak 
\fi 
\H@refstepcounter{Hfootnote}%
\hyper@makecurrent{Hfootnote}%
\hyper@linkstart{link}{\@currentHref}%
\@makefnmark 
\hyper@linkend 
\ifFN@pp@towrite 
\FN@pp@writetemp 
\FN@pp@towritefalse 
\fi 
\FN@mf@prepare 
\ifhmode\spacefactor\@x@sf\fi 
\relax%
}%
\fi 

\makeatother 

\theoremstyle{plain}
\newtheorem{theorem}{Theorem}[section]
\newtheorem{proposition}[theorem]{Proposition}
\newtheorem{lemma}[theorem]{Lemma}
\newtheorem{corollary}[theorem]{Corollary}
\newtheorem{conjecture}[theorem]{Conjecture}

\theoremstyle{definition}

\newtheorem{remark}[theorem]{Remark}

\newcommand{\D}{{\mathcal D}}

\newcommand{\A}{{\mathcal A}}

\newcommand{\GG}{{\mathcal G}}

\newcommand{\T}{{\mathcal T}}

\newcommand{\diam}{{\rm diam}}

\newcommand{\FF}{\mathcal F}

\newcommand{\mstar}{{\mbox{\raisebox{1px}{\scalebox{0.7}{$\star$}}}}}

\sectionfont{\large \sc}
\subsectionfont{\normalsize}

\long\def\symbolfootnote[#1]#2{\begingroup%
\def\thefootnote{\fnsymbol{footnote}}\footnote[#1]{#2}\endgroup}

\def\blfootnote{\xdef\@thefnmark{}\@footnotetext}

\begin{document}

{\Large \bfseries \sc Once punctured disks, non-convex polygons, and pointihedra}

{\bfseries Hugo Parlier\symbolfootnote[1]{\normalsize Research supported by Swiss National Science Foundation grants PP00P2\textunderscore 128557 and PP00P2\textunderscore 153024
}, Lionel Pournin\symbolfootnote[2]{\normalsize Research funded by Ville de Paris {\'E}mergences project ``Combinatoire {\`a} Paris''.\\
{\em Key words:} flip-graphs, triangulations of surfaces, combinatorial moduli spaces
}}

\vspace{0.2cm}
{\em Abstract.}
We explore several families of flip-graphs, all related to polygons or punctured polygons. In particular, we consider the topological flip-graphs of once-punctured polygons which, in turn, contain all possible geometric flip-graphs of polygons with a marked point as embedded sub-graphs. Our main focus is on the geometric properties of these graphs and how they relate to one another. In particular, we show that the embeddings between them are strongly convex (or, said otherwise, totally geodesic). We also find bounds on the diameters of these graphs, sometimes using the strongly convex embeddings. Finally, we show how these graphs relate to different polytopes, namely type D associahedra and a family of secondary polytopes which we call pointihedra.

\section{Introduction}\label{sec:intro}

Triangulations of surfaces are naturally linked to different areas of mathematics including combinatorics, graph theory, geometric topology and anything having to do with surface geometry. They serve as combinatorial models for geometric structures on surfaces, encode surface homeomorphisms and are related to the geometric positioning of points on two dimensional structures. Surface triangulations can be related to one another by {\it flip} transformations that, in their simplest manifestation, consist in exchanging two arcs (see the next section for a formal definition). Given a set of triangulations, we are interested in the geometry of the associated \emph{flip-graph}. The vertices of this graph are the triangulations and its edges link two triangulations whenever they differ by a single flip.

A well studied example is the graph $\A_n$ of the triangulations of a convex Euclidean $n$-gon: this graph turns out to be the $1$-skeleton of a polytope, the associahedron, and its diameter is by now well understood \cite{Pournin2014,SleatorTarjanThurston1988}. Flip-graphs of topological surfaces have also been studied: in this case, triangulations are considered up to isotopy, and the underlying flip-graphs are in general infinite graphs. They are related to the large scale geometry of the self-homeomorphism groups of the surface \cite{DisarloParlier2014}. Somewhere in between the general topological case and the graph of the associahedron lie the triangulations of {\it filling surfaces} \cite{ParlierPournin2014,ParlierPournin2015}. These are orientable topological surfaces but, like in the case of a polygon, one varies the number of marked points along a privileged boundary curve. When the filling surface is a disk, this gives rise to the graph of the associahedron, because there is a one-to-one correspondence between the triangulations of a disk with $n$ marked points on the boundary, considered up to isotopy, and the Euclidean geodesic triangulations of a fixed convex $n$-gon. The only other instance that gives rise to a finite flip-graph is when the filling surface is a punctured disk; this particular graph will be studied in the sequel and for future reference we call it $\T_n$. There is a first natural relationship between these two graphs: considering triangulations that contain a fixed arc from the puncture (which we consider as an interior vertex) to an outer vertex, it is not difficult to see that there are multiple copies of $\A_{n+2}$ inside $\T_n$. Our first main result concerns the geometry of $\T_n$: we are able to identify its diameter exactly.

\begin{theorem}
For every positive integer $n$, $\diam(\T_n) = 2n-2$.
\end{theorem}

The methods we use to obtain the lower bound on this diameter are similar to those in \cite{ParlierPournin2014} and \cite{ ParlierPournin2015}, but the specific nature of the once-punctured disk provides a particularly nice illustration of them. It is interesting to note that the graph $\mathcal{D}_n$ of the related type D associahedron admits $\T_n$ as a proper subgraph. As shown in \cite{CeballosPilaud2016}, the diameter of $\mathcal{D}_n$ is also $2n-2$. Although neither of these two diameter results imply the other, the similarity is striking. In particular, the  diametrically opposite vertices of $\mathcal{D}_n$ exhibited in \cite{CeballosPilaud2016} do not both belong to $\T_n$ so we need to find new pairs of vertices to prove our lower bounds on the diameter of $\T_n$. 

Another natural variant on $\A_n$ is to relax the convexity condition of the underlying polygon: topologically nothing has changed but here we are looking for {\it geometric} triangulations of the polygon. This gives rise to natural connected subgraphs of $\A_n$. Note that their connectedness is not a priori obvious \cite{HurtadoNoyUrrutia1999}. Similarly, one can take a Euclidean convex $n$-gon $P$ and place a marked point (referred to as a puncture) in its interior asking again for the geometric triangulations of the resulting punctured polygon $P^\mstar$. This time, the obtained flip-graph $\mathcal{F}(P^\mstar)$ is a natural subgraph of $\T_n$. In both graphs, there are copies of flip-graphs of polygons. In a multitude of different variants, we prove {\it convexity} results about how these subgraphs lie in the larger graphs.  We show that certain embedded subgraphs are strongly convex by which we mean that a geodesic path between any two vertices of the subgraph is entirely contained in the subgraph. This allows us to find lower bounds on diameter of the different graphs. As an example of our results we show the following.

\begin{theorem}
Let $P$ be a convex $n$-gon. For any placement of the interior marked point within $P$, the resulting punctured polygon has a flip-graph of diameter at most $2n-6$. In addition, one can place the interior marked point within $P$ so that this punctured polygon has a flip-graph of diameter at least $2n-8$. 
\end{theorem}

The last flip-graphs we consider are slightly different from the above graphs; they are obtained as the $1$-skeleton of some secondary polytopes \cite{GelfandKapranovZelevinsky1991}, just like $\A_n$ can be obtained as the $1$-skeleton of an associahedron. These graphs are closely related to the ones described above. Let $P$ be a convex Euclidean $n$-gon and $P^\mstar$ the same polygon with an interior marked point. The graph we consider contains both the flip-graph of $P$ (which is isomorphic to $\A_n$) and the flip-graph of $P^\mstar$: in particular its vertex set is the union of the two vertex sets. In addition there are edges corresponding to extra flip moves between triangulations of $P$ and triangulations of $P^\mstar$ that consist in inserting the puncture inside a triangulation (a proper definition is given in the next section). The resulting flip-graph, denoted by ${\overline{\mathcal F}}(P^{\mstar})$ is the graph of a polytope which we call a pointihedron. Our main results are, again, bounds on the diameter of this graph and convexity results on how $\A_n$ and $\mathcal{F}(P^{\mstar})$ lie within it at least for certain placements of the puncture. For instance:

\begin{theorem}\label{thm:SCP}
Let $P$ be a convex $n$-gon. For certain placements of the puncture, the natural embedding $\A_n \hookrightarrow {\overline{\mathcal F}}(P^{\mstar})$ is strongly convex.
\end{theorem}

We conjecture that this theorem remains true for any placement of the puncture. However, the methods we use to prove it do not seem to extendable to the general case. 
Finally, it is shown that, for some placement of the puncture, the natural embedding $\mathcal{F}(P^{\mstar})\hookrightarrow {\overline{\mathcal F}}(P^{\mstar})$ is not convex.

\textbf{Acknowledgements.} We are grateful for conversations with Delphine Milenovic who studied a particular variation of the geometric punctured flip-graph in her master's thesis. In particular she coined the term ``pointihedra''. Part of this work was done while the first author was visiting the LIPN at the University Paris 13 and the authors thank them for their support.  

\section{Preliminaries}

We begin the section by defining the topological flip-graphs of a disk and a once punctured disk and discuss some of their geometric interpretations. We then move on to defining geometric flip-graphs, some of which are intrinsically geometric as, for instance, the flip-graphs of non-convex or punctured Euclidean polygons. The last flip-graph we describe in this section, related to the geometric ones, turns out to be the $1$-skeleton of a polytope, the \emph{pointihedron}. Along the way, we explain where the simplicial embeddings come from and we conclude the section with some general remarks on simplicial embeddings between flip-graphs.

\subsection{Topological flip-graphs}

Let $\Delta$ be the disk and $\Delta^\mstar$ the punctured disk where the puncture is viewed as a marked point. If $\Sigma$ is either of these surfaces, we denote by $\Sigma_n$ the surface obtained from it by placing $n$ marked points on the boundary curve. Two arcs between marked points of $\Sigma_n$ are non-crossing when they have a common endpoint or disjoint interiors. A triangulation of $\Sigma_n$ is a maximal set of such non-crossing arcs. We denote by $\FF(\Sigma_n)$ the flip-graph of $\Sigma_n$. The vertices of this graph are the isotopy classes of the triangulations of $\Sigma_n$ and its edges connect triangulations that differ by a single arc. Equivalently, two triangulations are connected by an edge of $\FF(\Sigma_n)$ when they can be obtained from one another by a single {\it flip} operation. This operation consists in glueing two neighboring triangles into a quadrilateral, and then split it again into two triangles in the unique other possible way, or equivalently, in exchanging the two diagonals of this quadrilateral.

We denote
$$
\A_n=\FF(\Delta_n)
$$
as this graph is isomorphic to the $1$-skeleton of the associahedron \cite{Lee1989}. In particular, note that in this case $\FF(\Delta_n)$ can be given a geometric interpretation because its vertices are in one-to-one correspondence with the Euclidean triangulations of any convex $n$-gon.

For future reference we denote
$$
\T_n=\FF(\Delta^{\mstar}_n)
$$
This graph also has a geometric interpretation, but in terms of hyperbolic geometry. Take a hyperbolic ideal polygon $H_n$ with $n$ ideal points and with an interior cusp. The graph $\T_n$ then can be seen as the set of geodesic ideal triangulations of $H_n$ with the flip relations simply because every isotopy class of a triangulation is realized by a unique ideal geodesic triangulation. Although we will not make specific use of this model for the flip-graph, it is always good to keep in mind that geometric interpretations aren't always Euclidean. However, in the remainder of the article, we will use the word \emph{geometric} for triangulations of Euclidean objects.

We close this discussion on $\T_n$ by observing that it is closely related to the graph of the type D associahedron \cite{CeballosPilaud2016}, which we denote by $\D_n$. We refer to \cite{CeballosPilaud2016} for a precise definition of $\D_n$, but briefly said, although the vertices of these graphs consist in collections of arcs of the same cardinality, $\D_n$ differs from $\T_n$ in that multiple loops around the puncture are allowed to coexist. There is also a flip operation allowing to pass from one such loop to another. It is interesting to note that $\T_n$ is a proper subgraph of $\D_n$, a polytopal graph, and as we will discover in Section \ref{sec:punctureddisk}, they both have the same diameter.  

\subsection{Geometric flip-graphs}

Let $P$ be a possibly non-convex simple Euclidean $n$-gon. As in the topological case, a triangulation of $P$ is a maximal set of pairwise non-crossing arcs, with the only difference that the arcs we choose from are the Euclidean line segments contained in $P$ and connecting two vertices of $P$. As mentioned above, the triangulations of $\Delta_n$ can be viewed as the triangulations of any convex $n$-gon. This property fails with non-convex polygons: some segment between two vertices of such a polygon $P$ lies outside of it and thus the triangulations of $\Delta_n$ containing the corresponding arc cannot be realized as triangulations of $P$. We denote by $\FF(P)$ the geometric flip-graph of $P$, whose vertices are the triangulations of $P$ and whose edges connect any two triangulations that differ by a single arc. Note here that flips, as operations, consist in exchanging the diagonals of a convex quadrilateral. Convexity makes it sure that the exchanged arcs are both contained in the polygon.

As shown in \cite{HurtadoNoyUrrutia1999}, $\FF(P)$ is connected. Note that $\FF(P)$ is, in general, a subgraph of $\A_n$. In other terms, if $P$ is a simple Euclidean $n$-gon, we have a natural map 
$$
\FF(P) \hookrightarrow \A_n
$$
where $\hookrightarrow$ denotes a simplicial injection.

We now turn to the case of punctured polygons. We consider a fixed convex Euclidean $n$-gon $P$ and denote by $P^\mstar$ a geometric ``punctured" version of it. So $P^\mstar$ is really a choice of a $(n+1)$-th marked point $p$ in the interior of $P$. A triangulation of $P^\mstar$ is, again, a maximal set of pairwise non-crossing arcs. In this case, however, the arcs we choose from are the Euclidean line segments whose two endpoints are either two vertices of $P$, or a vertex of $P$ and the puncture $p$. Note that, by the maximality of a triangulation as a set of edges, all marked points including $p$ are vertices of it. We denote by $\FF(P^\mstar)$ the corresponding geometric flip-graph, that can be defined exactly as in the above two cases. In particular, as an operation, a flip still exchanges the diagonals of convex quadrilaterals.

Triangulations of $\Delta^\mstar_n$ cannot all be realized as triangulations of $P^\mstar$ either. In addition, $\FF(P^\mstar)$, as a subgraph of $\T_n$, is always proper: it is impossible to choose $P^\mstar$ in such a way that every topological arc can be realized geometrically. To see this, note that there are two topological arcs between any pair of outer marked points on $\Delta^\mstar$, depending on which side of the puncture they lie (see Fig. \ref{Pfigure.2.1}). 
\begin{figure}
\begin{centering}
\includegraphics{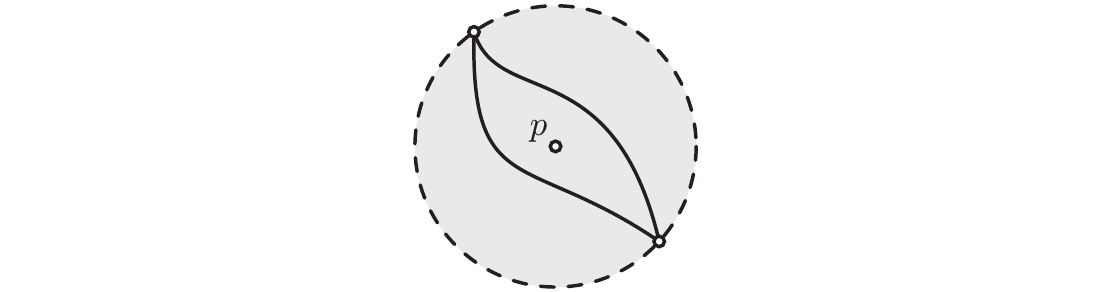}
\caption{Two non-isotopic arcs with the same vertex set.}\label{Pfigure.2.1}
\end{centering}
\end{figure}
However, at most one of them can be realized as an arc of a triangulation of $P^\mstar$ .

In summary, for every choice of $P^\mstar$, we get a simplicial embedding
$$\FF(P^\mstar) \hookrightarrow \T_n$$
which is never onto. 

\subsection{Pointihedra}

The last flip-graph we consider in this article can be obtained as the $1$-skeleton of a polytope, just as $\FF(\Delta_n)$ can be obtained as the $1$-skeleton of an associahedron, and for this reason, it is of special interest.

Let $P$ be a convex polygon and $P^\mstar$ the same polygon with a (fixed) marked point in its interior, as above. We need to define flips that insert the puncture inside a triangulation of $P$ (or, in an equivalent way, delete the puncture from a triangulation of $P^\mstar$). Consider a triangulation $T$ of $P$. If the puncture lies in the interior of a triangle $t$ of $T$, then one can flip the puncture into $T$ by introducing the three edges that connect it to the vertices of $t$, resulting in a triangulation of $P^\mstar$. If the puncture lies on an edge $\varepsilon$ of $T$, then call $q$ the quadrilateral obtained by glueing the two triangles of $T$ incident to $\varepsilon$. The puncture can be flipped into $T$ by first removing $\varepsilon$ from $T$ and then re-triangulating $q$ with the four edges that connect the puncture to its vertices, which also results in a triangulation of $P^\mstar$. We denote by ${\overline{\mathcal F}}(P^{\mstar})$ the graph obtained by first taking the disjoint union of $\FF(P)$ and $\FF(P^\mstar)$, and by then adding an edge between a triangulation of $P$ and a triangulation of $P^\mstar$ whenever the latter can be obtained by flipping the puncture into the former.

As mentioned above, there are polytopes whose graph is isomorphic to ${\overline{\mathcal F}}(P^{\mstar})$. We will call any such polytope a \emph{pointihedron}. For instance, consider the set of points made up of $p$ together with all the vertices of $P$. The secondary polytope of this set is a pointihedron, just as the secondary polytope of the vertex set of $P$ is an associahedron. We refer the interested reader to \cite{GelfandKapranovZelevinsky1991,DeLoeraRambauSantos2010} for more details about secondary polytopes. Note that, while we have two natural embeddings
$$
\FF(P^\mstar) \hookrightarrow {\overline{\mathcal F}}(P^{\mstar})\mbox{,}
$$
and
$$
\FF(P) \hookrightarrow {\overline{\mathcal F}}(P^{\mstar})\mbox{,}
$$
only $\FF(P)$ will be shown to be a geodesic subgraph of ${\overline{\mathcal F}}(P^{\mstar})$ for some placements of the puncture, or in the terminology we now introduce, only the latter embedding is strongly convex for such placements. 

If $Y$ is a graph, we say that a subgraph $X$ of $Y$ is {\it strongly convex in $Y$} if all the geodesics between any two vertices of $X$ lie entirely in $X$. Similarly an embedding $X \hookrightarrow Y$ between graphs is said to be strongly convex if the image of $X$ is strongly convex inside $Y$. In particular this implies that the intrinsic and extrinsic geometries of $X$ coincide. One of the key ingredients we shall use in the sequel is the strong convexity of many of the embeddings we introduced above. 

\subsection{Remarks on embeddings between flip-graphs}\label{Psection.2.4}

Although we will not make explicit use of this here, it is interesting to note that embeddings between topological flip-graphs are pretty well understood. The most natural way of getting an embedding is by taking all the triangulations of a surface that contain a fixed set of arcs. This provides a copy of the flip-graph of the surface cut along this set of arcs. Except for some low complexity situations, all flip-graph embeddings arise in this way \cite{AramayonaKoberdaParlier2014}. This works for both finite and infinite flip-graphs and tells, for example, that whenever there is a copy of the flip-graph of, say a $17$-gon inside the flip-graph of a $31$-gon, there is a collection of $14$ arcs that belong to each and every triangulation in this copy. 

In \cite{DisarloParlier2014}, it is shown that, in full generality, these natural copies of flip-graphs are strongly convex. Specifically let $\mu$ be a multi-arc on a surface $\Gamma$, $\FF(\Gamma)$ the flip-graph of $\Gamma$ and $\FF(\Gamma\setminus \mu)$ is the flip-graph of $\Gamma\setminus \mu$. The natural simplicial embedding 
$$
\FF(\Gamma\setminus \mu) \hookrightarrow\FF(\Gamma)\mbox{,}
$$
corresponding to looking at all triangulations containing $\mu$, is strongly convex.

\section{The punctured disk}\label{sec:punctureddisk}

We obtain the diameter of $\FF(\Delta^\mstar_n)$ in this section using techniques already used in \cite{ParlierPournin2014,ParlierPournin2015,Pournin2014,SleatorTarjanThurston1988}. We will not need to use these techniques in all their generality here and will only describe the needed notions along the way. The interested reader is referred to the above mentioned articles for details about them. The claimed result will be obtained by giving matching lower and upper bounds on $\diam(\FF(\Delta^\mstar_n))$. The proof of the upper bound is straightforward and we give it first.

\begin{lemma}\label{Plemma.4.1}
Any two triangulations of the once punctured disk $\Delta^\mstar_n$ can be transformed into one another using at most $2n-2$ flips.
\end{lemma}
\begin{proof}
Consider a triangulation $T$ of $\Delta^\mstar_n$, and recall that it has exactly $n$ interior arcs. If the puncture is not incident to all the interior arcs of $T$, it is always possible to introduce one additional such arc by performing a single flip. Moreover, the puncture is already incident to at least one interior arc of $T$. Hence, it is possible to transform $T$ into the triangulation of $\Delta^\mstar_n$ whose interior arcs are all incident to the puncture by at most $n-1$ flips. Therefore, any two triangulations of $\Delta^\mstar_n$ can be transformed into one another using at most $2n-2$ flips, as claimed².
\end{proof}

The lower bound will be obtained by showing that the two triangulations $A_n^-$ and $A_n^+$ depicted in Fig. \ref{Pfigure.4.1} have distance at least $2n-2$ in $\FF(\Delta^\mstar_n)$. Most of the interior arcs in these triangulations are arranged as a zigzag. The puncture is placed at one end of the zigzag in $A_n^-$ and at the other end in $A_n^+$. It is linked to the boundary by a single interior edge and surrounded by a loop twice incident to the boundary. Note that, when $n$ is equal to $1$, $2$, or $3$, the depiction in Fig. \ref{Pfigure.4.1} may be ambiguous.
\begin{figure}
\begin{centering}
\includegraphics{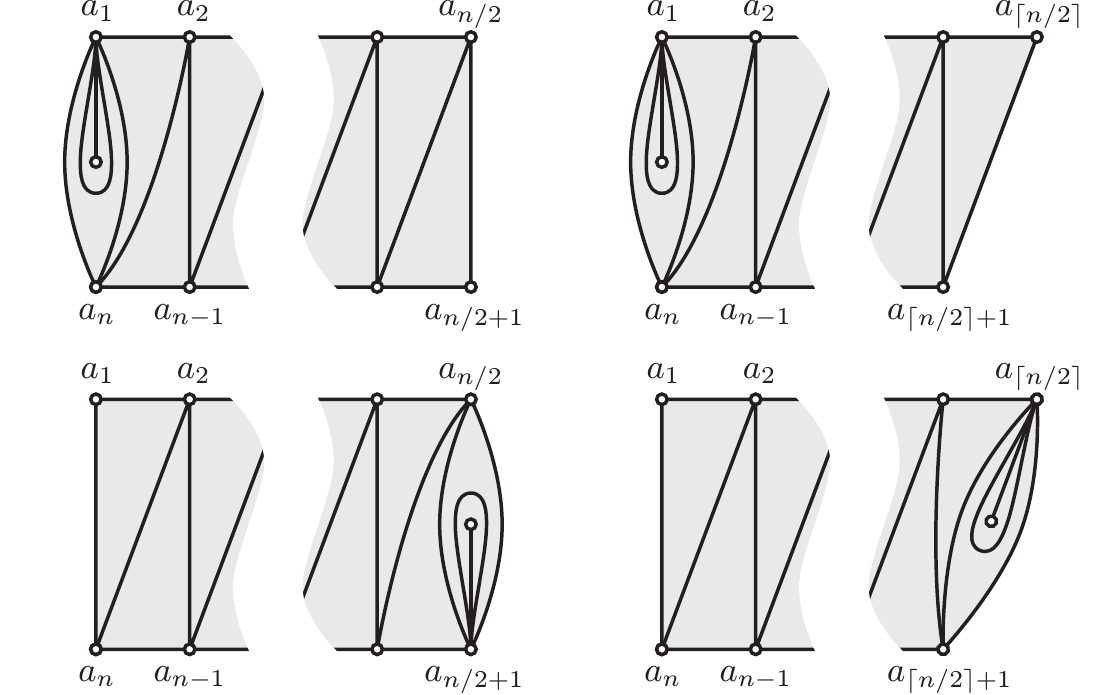}
\caption{The triangulations $A_n^-$ (top row) and $A_n^+$ (bottom row) of the once punctured disk, depicted when $n$ is even (left) and odd (right).}\label{Pfigure.4.1}
\end{centering}
\end{figure}
For this reason, $A_n^-$ and $A_n^+$ are shown in Fig. \ref{Pfigure.4.2} when $1\leq{n}\leq3$.

\begin{lemma}\label{Plemma.4.2}
At least $2n-2$ flips are required to transform $A_n^-$ into $A_n^+$.
\end{lemma}
\begin{proof}
The proof proceeds by induction on $n$. When $n=1$, $A_n^-$ and $A_n^+$ are identical as shown in Fig. \ref{Pfigure.4.2}. Hence, they have distance $0$ in $\FF(\Delta^\mstar_n)$ which is equal to $2n-2$ in this case. Assume that the result holds for some $n\geq1$.

Consider the triangulation $U$ of $\Delta^\mstar_n$ obtained from $A_{n+1}^-$ by displacing vertex $a_{n+1}$ to $a_1$ within the boundary, while preserving the incidences between a vertex and an arc. Note that, after this operation the loop that was twice incident to $a_1$ and the interior arc whose endpoints were $a_1$ and $a_{n+1}$ have become isotopic. Therefore, we assume that one of them has been removed from $U$. 
\begin{figure}
\begin{centering}
\includegraphics{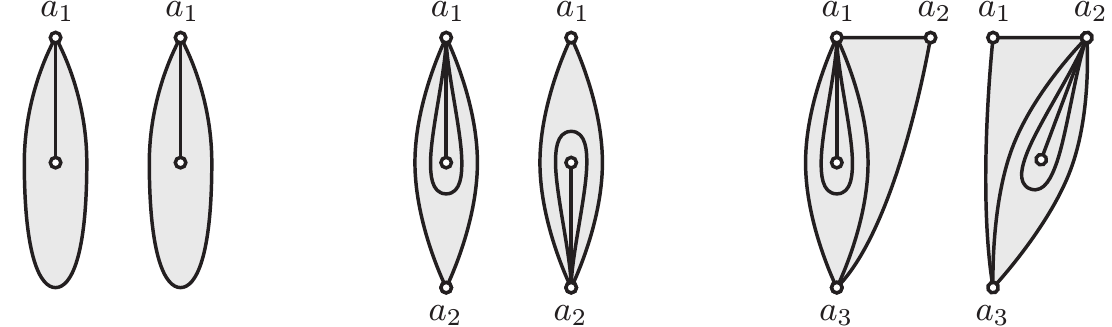}
\caption{Triangulations $A_1^-$, $A_1^+$, $A_2^-$, $A_2^+$, $A_3^-$, and $A_3^+$ (from left to right).}\label{Pfigure.4.2}
\end{centering}
\end{figure}
This operation is referred to in \cite{ParlierPournin2014,ParlierPournin2015,Pournin2014} as the \emph{deletion of vertex $a_{n+1}$ from $A_{n+1}^-$}. Indeed, the identification of vertices $a_1$ and $a_{n+1}$ can be thought of as the removal of the latter. Performing the same operation inside $A_n^+$ results in a triangulation $V$ of $\Delta^\mstar_n$. Two arcs are made isotopic in this case as well, that cannot both be kept in $V$: the boundary arc between $a_1$ and $a_2$ and the interior arc of $A_{n+1}^+$ with vertices $a_2$ and $a_n$. 

Observe that relabelling vertex $a_i$ by $a_{n-i+2}$ within $U$ and $V$ for all $i$ so that $2\leq{i}\leq{n}$ results in $A_n^-$ and $A_n^+$ respectively. In particular, the number of flips required to transform $U$ into $V$ is also the number of flips required to transform $A_n^-$ into $A_n^+$. We now investigate how the deletion of vertex $a_{n+1}$ affects a sequence of flips that changes triangulations $A_{n+1}^-$ and $A_{n+1}^+$ into one another.

Consider a sequence $T_0$, ..., $T_k$ of triangulations of $\Delta^\mstar_{n+1}$ such that $T_0=A_{n+1}^-$, $T_k=A_{n+1}^+$, and whenever $1\leq{i}\leq{k}$, $T_{i-1}$ can be transformed into $T_i$ by a flip. This sequence of triangulations can be thought of as a sequence of $k$ flips that transforms $A_{n+1}^-$ into $A_{n+1}^+$. We will assume that $k$ is the least number of flips required to do so. Deleting  $a_{n+1}$ from triangulations $T_0$, ..., $T_k$ results in a sequence of triangulations that starts with $U$, ends with $V$, and so that two consecutive triangulations are either related by a flip, or identical. Indeed, consider the quadrilateral $q$ modified by the flip that changes $T_{i-1}$ into $T_i$. If this flip affects the triangle incident to the boundary arc $\alpha$ with endpoints $a_{n+1}$ and $a_1$ (that is, if $\alpha$ is an edge of $q$), then the deletion changes $q$ into a triangle. In this case, $T_{i-1}$ is identical to $T_i$. If, on the contrary, $\alpha$ is not an edge of $q$, then the two triangulations obtained by deleting $a_{n+1}$ in $T_{i-1}$ and in $T_i$ are still related by the flip that modifies $q$. In particular, deleting $a_{n+1}$ from $T_0$, ..., $T_k$ provides a sequence of $k-l$ flips that transforms $U$ into $V$, where $l$ is the number of flips that modify the triangle incident to $\alpha$ in the original sequence. Hence, the number of flips required to transform $U$ into $V$, or equivalently $A_n^-$ into $A_n^+$, is at most $k-l$. This number of flips is, by induction, at least $2n-2$. As a consequence,
\begin{equation}\label{Plemma.4.2.equation.1}
2n-2\leq{k-l}\mbox{.}
\end{equation}

Call $t^-$ and $t^+$ the triangles incident to $\alpha$ in respectively $A_{n+1}^-$ and $A_{n+1}^+$. As $t^-$ and $t^+$ are distinct, $l$ must be positive. It turns out that $l$ is necessarily greater than $1$. Indeed, otherwise, a single flip replaces $t^-$ by $t^+$ along our sequence of flips, as shown in Fig. \ref{Pfigure.4.3}.
\begin{figure}
\begin{centering}
\includegraphics{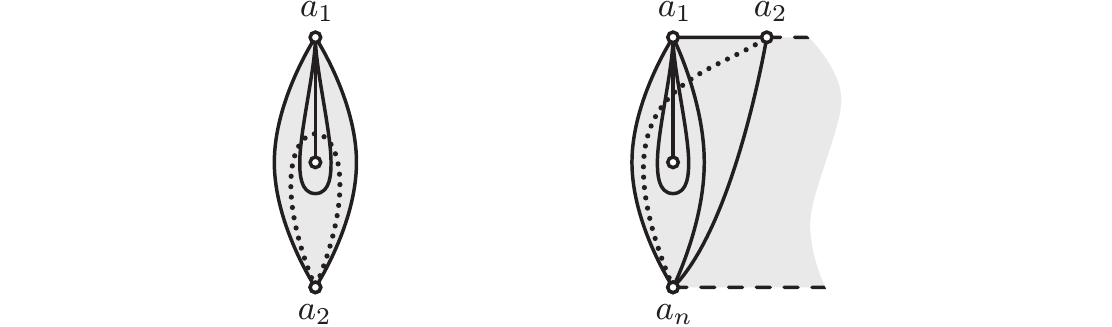}
\caption{A flip that replaces $t^-$ by $t^+$, depicted when $n=2$ (left) and when $n>2$ (right). The introduced edge is dotted.}\label{Pfigure.4.3}
\end{centering}
\end{figure}
As can be seen, the arc introduced by this flip crosses two arcs of the triangulation it is performed in: the arc incident to the puncture and an edge of $t^-$ (which is crossed twice). Hence, such a flip is impossible, and $l$ must be greater than $1$. According to inequality (\ref{Plemma.4.2.equation.1}), the number of flips required to transform $A_{n+1}^-$ into $A_{n+1}^+$ is then at least $2n$, which completes the proof. 
\end{proof}

The following theorem is a direct consequence of the last two lemmas.

\begin{theorem}
For every positive integer $n$, $\diam(\FF(\Delta^\mstar_n)) = 2n-2$.
\end{theorem}

\begin{remark}
As mentioned in the preliminaries, $\T_n$ is a proper subgraph of $\D_n$ the type D associahedron graph. In \cite{CeballosPilaud2016} it is shown that
$$
\diam\left(\D_n\right) = 2n-2
$$
but the vertices used to prove this lower bound are very different from the ones we exhibit. In fact one of them does not belong to $\T_n$ and the other is distance at most $n-1$ from all vertices of $\T_n$. It seems plausible, though, that the diametrically opposite vertices we exhibit are still diametrically opposite in $\D_n$.
\end{remark}

\section{Geometric cases}

In this section, we consider the flip-graphs of Euclidean polygons or Euclidean punctured polygons. We begin with non-convex polygons. 

\subsection{Flip-graphs of non-convex polygons}

If $P$ is a simple, non-convex $n$-gon, the unique geodesic arc between two different vertices on the boundary is possibly not entirely contained in the polygon. In $\FF(P)$ we only consider triangulations made of arcs entirely contained in the polygon and as such, $\FF(P)$ is the subgraph of $\A_n$ induced by triangulations that do not contain any forbidden arc. There are many different possible non-convex polygons but we will only look at $n$-gons with $n-1$ vertices in convex position. 

We begin by observing the following. 

\begin{proposition}\label{Propos.nc}
Consider a simple $n$-gon $P$. If there is a unique pair of vertices of $P$ so that the segment between them is not contained in $P$, then $\diam(\FF(P)) \geq 2n -10$. 
\end{proposition}
\begin{proof}
First note that if $n$ is less than $5$, then $2n-10$ is negative and the result is immediate. We therefore assume that $n$ is at least $5$.

Explicit examples of triangulations in $\A_n$ at distance at least $2n-10$ are given in \cite{Pournin2014}. It is easy to see that they also belong to $\FF(P)$: indeed, $\FF(P)$ contains all the triangulations of $\A_n$ with at least one interior edge incident to the non-convex vertex $v$ of $P$. The two triangulations from \cite{Pournin2014} each have exactly two vertices that are not incident to any of their interior edges. As $n\geq5$, one can therefore place them in $P$ such that they each have at least one interior edge incident to $v$.

As $\FF(P)$ is a subgraph of $\A_n$, the distance of these triangulations in the former graph cannot be less than their distance in the latter.
\end{proof}

In general we have the following slightly weaker lower bound.

\begin{proposition}\label{Prop.NCV}
If $P$ is a simple polygon with $n-1$ vertices in convex position then
$$
\diam(\FF(P)) \geq 2n -16\mbox{.}
$$ 
\end{proposition}
\begin{proof}
Any such polygon can be cut along an arc $\alpha$ between two outer marked points such that the two resulting polygons are at worst ``slightly" non-convex in the sense of the previous theorem. Suppose that the resulting polygons are of size $n^-$ and $n^+$: we have $n^-+n^+ = n+2$. The flip-graph of all triangulations containing $\alpha$ is strongly convex inside $\A_n$ \cite{SleatorTarjanThurston1988}. According to Proposition \ref{Propos.nc}, the diameter of the individual flip-graphs are bounded below by $2n^- - 10$ and $2n^+-10$, so the diameter of $\FF(P)$ is bounded below by
$$
2 (n^-+n^+) - 20 = 2 n - 16\mbox{,}
$$
as desired.
\end{proof}

\subsection{Punctured Euclidean polygons}

We now explore the flip-graph of a convex $n$-gon $P^\mstar$ with a single marked point $p$ in its interior. As explained previously, for every $n\geq 3$, we can think of this choice as being the choice of $p$ within a fixed regular Euclidean $n$-gon. We are mainly interested in global properties (that do not depend on the choice of $p$) but some specific properties for certain placements of $p$ will be given as well. 

We recall from the preliminary section that for any choice of $P^\mstar$ there is a natural simplicial embedding $\FF(P^\mstar) \hookrightarrow \T_n$
which is never onto. We show that this embedding has a geometric meaning.

\begin{theorem}
The embedding $\FF(P^\mstar) \hookrightarrow \T_n$ is strongly convex.
\end{theorem}

\begin{proof}
We will define a projection map $\pi$ from $\T_n$ to $\FF(P^\mstar)$ which acts as the identity on $\FF(P^\mstar) $, and sends points at distance $1$ to points at distance at most $1$. Before defining the map formally, let us describe a physical model of it. Suppose you represent your polygon by putting nails on a wooden board and then place an extra nail in place of the interior marked point. Now take a triangulation of $\Delta^\mstar_n$ (or, equivalently, a vertex of $\T_n$). The image of this triangulation by $\pi$ is what happens when you make each of its arcs a rubber band and string it between the nails that correspond to its vertices.

We now define map $\pi$ in a formal way. We begin by looking at how it affects arcs. Let $\gamma$ be an arc between two (not necessarily distinct) marked points $a$ and $b$ of $\Delta^\mstar_n$. If $\gamma$ can be represented in $P^\mstar$ by a Euclidean line segment with the correct topology (this is the case, for instance, when either $a$ or $b$ is the interior marked point) then we define $\pi(\gamma) $ to be this segment. If it cannot, we define a unique geometric image in $P^\mstar$ as a multi-arc. Let $\alpha$ be the unique Euclidean arc between $a$ and the interior marked point $p$ of $P^\mstar$ and $\beta$ be the unique Euclidean arc between $b$ and $p$. We define $\pi(\gamma) = \alpha \cup \beta$. Note that if $a=b$ then $\alpha=\beta$. 

We now need to show that this map sends triangulations to triangulations. To see this, note that the image of a triangle is always a triangle, possibly with an extra arc from the interior marked point to one its vertices on the boundary. In particular, this means the complement of the image of a triangulation cannot contain any polygon of higher complexity, and so $\pi(T)$ is a triangulation for any $T\in \T_n$.

Triangulations at distance $1$ are mapped by $\pi$ to triangulations at distance at most $1$ as claimed. This is obvious since this map cannot increase the number of intersections between the arcs of two triangulations.
\end{proof}

As a consequence of this, we have the following corollary.

\begin{corollary}\label{Pcorollary.5.2}
Let $T$ and $T'$ belong to $\FF(P^\mstar)$ with $\alpha \in T\cap T'$ an arc. Then all triangulations along a geodesic path between $T$ and $T'$ contain $\alpha$.
\end{corollary}

\begin{proof}
As discussed in the preliminaries, this is true for $\T_n$. As the embedding of $ \FF(P^\mstar)$ inside $\T_n$ is strongly convex, this remains true for $\FF(P^\mstar)$.
\end{proof}

Consider the subgraph $\GG$ induced in $\FF(P^\mstar)$ by the triangulations that share a given multi-arc $\mu$. Corollary \ref{Pcorollary.5.2} could also be formulated as the strong convexity of the natural embedding $\GG\hookrightarrow\FF(P^\mstar)$. This convexity property immediately bounds the diameter of $\FF(P^\mstar)$ below by that of $\GG$. If $\mu$ splits $P^\mstar$ into smaller polygons $P_1$, ... $P_k$, then the latter is in turn bounded below as:
$$
\sum_{i=1}^k\diam(\FF(P_i))\leq\diam(\GG)\mbox{.}
$$

From this observation, we obtain the following general lower bound that does not depend on the placement of the puncture.

\begin{theorem}\label{Ptheorem.5.3}
For any punctured convex $n$-gon $P^\mstar$,
$$
\diam(\FF(P^\mstar)) \geq 2n - 12\mbox{.}
$$
\end{theorem}

\begin{proof}
Call $p$ the puncture of $P^\mstar$. Let us consider two distinct arcs $\alpha$ and $\beta$, both incident to $p$. These arcs split $P^\mstar$ into two simple polygons, say $P^-$ and $P^+$. Respectively call $a$ and $b$ the vertices of $\alpha$ and $\beta$ that are distinct from $p$. One can choose $\alpha$ and $\beta$ in such a way that either $a$, $b$, and $p$ are collinear, or $P^+$ is convex and the only line segment between two vertices of $P^-$ that is not contained in $P^-$ is that between $a$ and $b$. The numbers $n^-$ and $n^+$ of vertices of $P^-$ and $P^+$ satisfy
\begin{equation}\label{Ptheorem.5.3.equation.1}
n^-+n^+ = n +4\mbox{.}
\end{equation}

By Proposition \ref{Prop.NCV}, the diameters of $\FF(P^-)$ and $\FF(P^+)$ are bounded below by respectively $2n^--10$ and $2n^+-10$. Now consider the subgraph induced in $\FF(P^\mstar)$ by the triangulations that contain both $\alpha$ and $\beta$. It follows from Corollary \ref{Pcorollary.5.2} that the embedding of this subgraph into $\FF(P^\mstar)$ is strongly convex. Therefore,
$$
\diam(\FF(P^\mstar))\geq2n^--10+2n^+-10\mbox{.}
$$
Combining this inequality with (\ref{Ptheorem.5.3.equation.1}) completes the proof.
\end{proof}

It turns out that Theorem \ref{Ptheorem.5.3} is nearly sharp:

\begin{lemma}\label{Plemma.5.4}
Any two triangulations of any punctured $n$-gon $P^\mstar$ can be transformed into one another using at most $2n-6$ flips.
\end{lemma} 
\begin{proof}
Let $P^\mstar$ be a punctured polygon. We can use a a proof similar to that of Lemma \ref{Plemma.4.1}. As the puncture is incident to at least three arcs of a triangulation of $P^\mstar$, the additive constant in the bound is indeed $-6$ instead of $-2$. 
\end{proof}

The gap between the bounds provided by Theorem \ref{Ptheorem.5.3} and Lemma \ref{Plemma.5.4} is small. For some placements of the puncture, this gap can be made even smaller using the result of \cite{Pournin2014} together with Corollary \ref{Pcorollary.5.2}.
\begin{theorem}\label{Ptheorem.5.5}
For any convex $n$-gon $P$, one can find a placement of the puncture so that the resulting punctured polygon $P^\mstar$ satisfies $\diam(\FF(P^\mstar))\geq2n-8$.
\end{theorem}
\begin{proof}
Let $P$ be a convex $n$-gon. If $n$ is equal to $3$ or $4$, then the result is immediate as $\diam(\FF(P^\mstar))$ is non-negative. We therefore assume that $n$ is at least $4$. According to \cite{Pournin2014}, one can find two triangulations $U$ and $V$ of $P$ whose distance in $\FF(P)$ is not less than $2n-10$. It can be required that some triangles $u$ and $v$ of respectively $U$ and $V$ have non-disjoint interiors and exactly one common vertex $a$. Indeed, the triangulations at distance at least $2n-10$ given in \cite{Pournin2014} (see Figs. 5 and 6 therein) have such triangles whenever $n\geq4$. Now place the puncture, which we denote by $p$, both in the interior of $u$ and in that of $v$. Call $U'$ and $V'$ the triangulations of $P^\mstar$ obtained by flipping $p$ into respectively $U$ and $V$.

Consider a sequence $T_0$, ..., $T_k$ of triangulations of $P^\mstar$ such that $T_0=U'$, $T_k=V'$, and whenever $1\leq{i}\leq{k}$, $T_{i-1}$ can be transformed into $T_i$ by a flip. We will use an operation similar to the vertex deletions performed in the proof of Lemma \ref{Plemma.4.2}. By Corollary \ref{Pcorollary.5.2}, any $T_i$ contains the arc $\alpha$ with vertices $a$ and $p$. Hence, we can remove this arc from $T_i$ by displacing $p$ to $a$ while keeping the incidences between a vertex and an arc unaffected. Note that two pairs of arcs of $T_i$ are made isotopic by this operation. Removing one arc from each of these pairs results in a triangulation of $P$. We refer to this operation as the \emph{contraction of $T_i$ along arc $\alpha$}. 

By construction contracting $U'$ and $V'$ along $\alpha$ respectively results in $U$ and $V$. As in the proof of Lemma \ref{Plemma.4.2}, the triangulations obtained from $T_{i-1}$ and $T_i$ by this contraction are either identical or related by a flip. In fact they are identical when the flip that transforms $T_{i-1}$ into $T_i$ affects one of the two triangles incident to $\alpha$. Hence, contracting $T_0$, ..., $T_k$ along $\alpha$ provides a sequence of $k-l$ flips that transform $U$ into $V$, where $l$ is the number of flips that affect a triangle incident to $\alpha$. As at least $2n-10$ flips are required to change $U$ into $V$,
\begin{equation}\label{Ptheorem.5.5.equation.1}
2n-10\leq{k-l}\mbox{.}
\end{equation}

Now recall that triangles $u$ and $v$ have a unique common vertex. It follows that a triangle of $U'$ incident to $\alpha$ cannot be a triangle of $V'$. As a single flip cannot affect both the triangles incident to $\alpha$, there must be at least two flips that affect a triangle incident to $\alpha$ along any sequence of flips that transforms $U'$ and $V'$. As a consequence, $l\geq2$ and, assuming that $k$ is the least number of flips needed to transform $U'$ into $V'$, the result follows from inequality (\ref{Ptheorem.5.5.equation.1}).
\end{proof}

\section{The geometry of the pointihedron}

Let $P^\mstar$ be a punctured polygon obtained by placing a puncture within a polygon $P$. Recall that ${\overline{\mathcal F}}(P^{\mstar})$ is obtained by considering both $\mathcal{F}(P)$ and $\mathcal{F}(P^{\mstar})$ and connecting them with the flips that introduce or remove the puncture. It turns out that these additional flips allow for an upper bound of $2n-7$ on the diameter of ${\overline{\mathcal F}}(P^{\mstar})$, which is less by one than our upper bound on $\diam(\FF(P^{\mstar}))$:

\begin{lemma}\label{Plemma.6.1}
For any punctured $n$-gon $P^\mstar$, 
$$
\diam({\overline{\mathcal F}}(P^{\mstar}))\leq\diam(\A_n)+3\mbox{.}
$$
\end{lemma} 
\begin{proof}
Let $P^\mstar$ be a punctured polygon obtained by placing a puncture $p$ within a $n$-gon $P$. Call $F$ the triangulation of $P^\mstar$ whose all arcs are incident to $p$. We will exhibit sequences of flips that transform two triangulations of $P$ or $P^\mstar$ into one another. We review three cases depending on whether each of them is a triangulation of $P$ or $P^\mstar$.

First consider two triangulations of $P^\mstar$. When the combined degree of $p$ within them is a least $7$, then these triangulations can be transformed into one another by at most $2n-7$ flips via $F$, by the same procedure as in the proofs of Lemmas \ref{Plemma.4.1} and \ref{Plemma.5.4}. Since $\diam(\mathcal{A}_n)\geq2n-10$, the result follows. If the combined degree of $p$ in the two triangulations is $6$ (it cannot be less), then they each can be transformed into a triangulation of $P$ by the flip that removes the puncture. Hence, their distance in ${\overline{\mathcal F}}(P^{\mstar})$ is at most $\diam(\A_n)+2$, and the result holds.

Now consider a triangulation $T$ of $P^\mstar$. We will transform it into an arbitrary triangulation of $P$. If the degree of $p$ in $T$ is at least $5$, then it can be changed into $F$ by at most $n-5$ flips. Any triangulation of $P$ can be transformed into $F$ by first flipping $p$ into it and then by introducing all the missing arcs incident to $p$ using at most $n-3$ flips. This shows that $T$ can be transformed into any triangulation of $P$ by at most $2n-7$ flips. Since $\diam(\mathcal{A}_n)\geq2n-10$, the result follows. If the degree of $p$ in $T$ is at most $4$, then it can be transformed into a triangulation of $P$ by at most $2$ flips. Hence the distance of $T$ in ${\overline{\mathcal F}}(P^{\mstar})$ to any triangulation of $P$ is at most $\diam(\A_n)+2$, which implies the desired result.

Finally, two triangulations of $P$ can, by definition, be transformed into one another by at most $\diam(\A_n)$ flips, which completes the proof.
\end{proof}

\begin{remark}
As $\A_n$ has diameter $2n-10$ whenever $n$ is at least $13$, it is a direct consequence of Lemma \ref{Plemma.6.1} that the diameter of ${\overline{\mathcal F}}(P^{\mstar})$ is at most $2n-7$ for any punctured polygon $P^\mstar$ with at least $13$ boundary vertices.
\end{remark}

Let us turn our attention to lower bounds. For some placements of the puncture, we obtain a lower bound on the diameter of the pointihedron that matches our lower bound on the diameter of $\FF(P^\mstar)$.

\begin{theorem}\label{Ptheorem.6.2}
For any convex $n$-gon $P$, one can find a placement of the puncture so that the resulting punctured polygon $P^\mstar$ satisfies $\diam(\overline{\FF}(P^\mstar))\geq2n-8$.
\end{theorem}
\begin{proof}
The proof will be similar to that of Theorem \ref{Ptheorem.5.5}, except that Corollary \ref{Pcorollary.5.2} cannot be invoked in this case as the puncture can appear or disappear along a sequence of flips. This will further restrict the placement of the puncture. Consider a convex $n$-gon $P$. If $n$ is equal to $3$ or $4$, then the result is immediate as $\diam(\overline{\FF}(P^\mstar))\geq0$. We therefore assume that $n$ is at least $4$.

As in the proof of Theorem \ref{Ptheorem.5.5}, we build two triangulations $U'$ and $V'$ of $P^\mstar$ by flipping the puncture into two triangulations $U$ and $V$ of $P$ whose distance in $\FF(P)$ is not less than $2n-10$. We require that some ear $u$ of $U$ (i.e. a triangle with two edges in the boundary of $P$) and some triangle $v$ of $V$ have non-disjoint interiors and exactly one common vertex $a$. Note that $a$ is then necessarily the vertex that is incident to both the edges of $u$ in the boundary of $P$. This requirement is satisfied, for instance, by the triangulations at distance at least $2n-10$ given in \cite{Pournin2014} (see Figs. 5 and 6 therein) as soon as $n\geq4$. Place the puncture, which we denote by $p$, both in the interior of $u$ and in that of $v$. Call $U'$ and $V'$ the triangulations of $P^\mstar$ obtained by flipping $p$ into respectively $U$ and $V$.

Note that, as $u$ is an ear of $U$, then the arc $\alpha$ between $a$ and $p$ belongs to all the triangulations of $P^\mstar$. This makes it possible to contract them along $\alpha$ as in the proof of Theorem \ref{Ptheorem.5.5}, even though we cannot invoke Corollary \ref{Pcorollary.5.2}.

Consider a sequence $T_0$, ..., $T_k$ of triangulations of $P$ or $P^\mstar$ such that $T_0=U'$, $T_k=V'$, and whenever $1\leq{i}\leq{k}$, $T_{i-1}$ can be transformed into $T_i$ by a flip. We will contract $T_i$ along $\alpha$ whenever $T_i$ is a triangulation of $P^\mstar$. In this case $\alpha$ belongs to $T_i$ and one can perform this contraction in exactly the same way as in the proof of Theorem \ref{Ptheorem.5.5}. By construction, this contraction transforms $U'$ and $V'$ in respectively $U$ and $V$. When $T_i$ is already a triangulation of $P$, we assume as a convention that the contraction along $\alpha$ does not affect it. The triangulations obtained from $T_{i-1}$ and $T_i$ by the deletion are either identical or related by a flip. As in the proof of Theorem \ref{Ptheorem.5.5}, they are identical when the flip that transforms $T_{i-1}$ into $T_i$ affects one of the two triangles incident to $\alpha$. In particular, this is the case when this flip either introduces the puncture or removes it. Hence, contracting $T_0$, ..., $T_k$ along $\alpha$ provides a sequence of $k-l$ flips that transforms $U$ into $V$, where $l$ is the number of flips that affect a triangle incident to $\alpha$. As at least $2n-10$ flips are required to change $U$ into $V$, the following inequality holds:
\begin{equation}\label{Ptheorem.6.2.equation.1}
2n-10\leq{k-l}\mbox{.}
\end{equation}

As argued in the proof of Theorem \ref{Ptheorem.5.5}, if the puncture belongs to all triangulations $T_0$, ..., $T_k$, then $l$ is at least $2$. Now, if the puncture is removed by some flip between two of these triangulations, then some other flip must reintroduce it later because $V'$ admits $p$ as a vertex. These two flips affect a triangle incident to $\alpha$, and $l$ must be at least $2$ in this case as well. As a consequence, assuming that $k$ is the least number of flips needed to transform $U'$ into $V'$, the result follows from inequality (\ref{Ptheorem.6.2.equation.1}).
\end{proof}

It seems the above can be improved. In fact we conjecture the following.

\begin{conjecture}
Consider a convex $n$-gon $P$ so that $n\geq13$. For some placement of the puncture in the interior of $P$, the resulting punctured polygon $P^\mstar$ satisfies
$$
\diam(\overline{\FF}(P^\mstar)) = 2n-7.
$$
\end{conjecture}

We point out that we do not know whether it would be reasonable to state this conjecture for any placement of the puncture. We now focus on the geometry of the natural embedding of the associahedron graph inside the pointihedron graph. Some of the arguments in the proof of Theorem \ref{Ptheorem.6.2} can be used to show that this embedding is strongly convex when the puncture is placed within a triangle that shares at least two edges with the considered convex polygon. In fact, we can do better. A boundary quadrilateral of a convex polygon is a quadrilateral that shares at least three of its edges with this polygon. We have the following convexity result when the puncture is placed within the interior of the intersection of two boundary quadrilaterals as sketched in Fig. \ref{Pfigure.5.1}.

\begin{theorem}\label{thm:SCP}
Consider a convex $n$-gon $P$. Consider two boundary quadrilaterals $q^-$ and $q^+$ of $P$ that share exactly three vertices. For any placement of the puncture in the interior of $q^-\cap{q^+}$, the resulting punctured polygon $P^\mstar$ is such that the natural embedding
$$
\A_n \hookrightarrow {\overline{\mathcal F}}(P^{\mstar})
$$
is strongly convex.
\end{theorem}

\begin{proof} Assume the puncture $p$ is placed in the interior of $q^-\cap{q^+}$. Call $u$, $v$, and $w$ the three common vertices of $q^-$ and $q^+$ such that $v$ is between $u$ and $w$, as shown in Fig. \ref{Pfigure.5.1}. Denote by $\alpha$ the Euclidean arc connecting $v$ to $p$ and by $\beta$ the one connecting $u$ to $w$. Consider a sequence $T_0$, ..., $T_k$ of triangulations of $P$ or $P^\mstar$ such that $T_{i-1}$ and $T_i$ are related by a flip for $1\leq{i}\leq{k}$. We also assume that both $T_0$ and $T_k$ are triangulations of $P$ and that $k$ is their distance in ${\overline{\mathcal F}}(P^{\mstar})$. We will prove that $T_i$ is a triangulation of $P$ whenever $0\leq{i}\leq{k}$. To do so we will need to use a contraction operation along arc $\alpha$, very much like in the proof of Theorem \ref{Ptheorem.6.2}.

We review two cases.

First assume that $p$ belongs to the interior of the triangle with vertices $u$, $v$, and $w$.
\begin{figure}[b]
\begin{centering}
\includegraphics{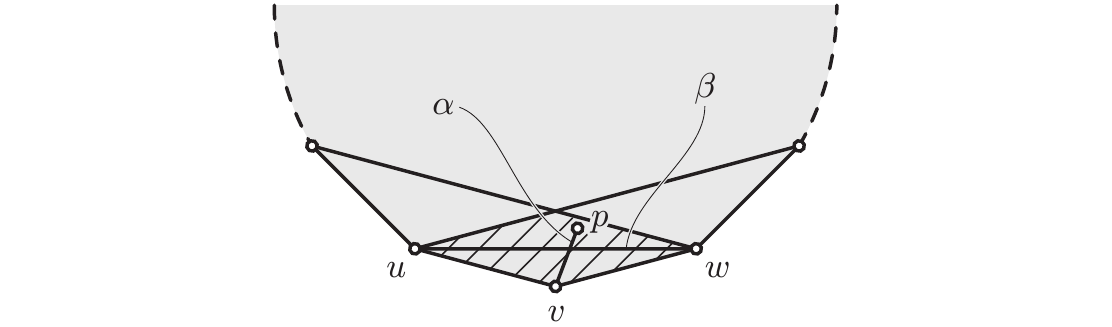}
\caption{The intersection (hatched) of two boundary quadrilaterals of a larger convex polygon. These quadrilaterals share three vertices $u$, $v$, and $w$. In this representation, the puncture $p$ is placed so that arcs $\alpha$ and $\beta$ are crossing.}\label{Pfigure.5.1}
\end{centering}
\end{figure}
Note that this triangle is actually a proper subset of $q^-\cap{q^+}$. Because of this particular placement of the puncture, the arcs that share their vertices with $P$ never cross $\alpha$. Therefore, $\alpha$ belongs to every triangulation of $P^\mstar$. In this case, the contraction along $\alpha$ can be defined as in the proof of Theorem \ref{Ptheorem.6.2}. By construction, the triangulations $T'_{i-1}$ and $T'_i$ obtained by contracting $T_{i-1}$ and $T_i$ along $\alpha$ are either identical or related by a flip. More precisely, $T_{i-1}$ and $T_i$ are identical when the flip that transforms them into one another affects a triangle incident to arc $\alpha$. Since $k$ is the least number of flips needed to transform $T_0$ into $T_k$, two consecutive triangulations in the sequence $T'_0$, ..., $T'_k$ cannot be identical because this sequence of triangulations would otherwise provide a sequence of less than $k$ flips that transforms $T_0$ into $T_k$. As any flip that introduces the puncture into a triangulation affects a triangle incident to arc $\alpha$, then such a flip cannot be the one that transforms $T_{i-1}$ into $T_i$ and the result follows.

Now assume that $p$ does not belong to the interior of the triangle with vertices $u$, $v$, and $w$. In this case, we define the contraction of $T_i$ along $\alpha$ as above when $T_i$ is a triangulation of $P$ or when $T_i$ contains $\alpha$. If $T_i$ is a triangulation of $P^\mstar$ that does not contain $\alpha$, this operation is still well defined because of our choice of placement of $p$. As $p$ lies in the interior of $q^-\cap{q^+}$, if $T_i$ does not contain $\alpha$, it must contain an arc that crosses $\alpha$. The arc $\beta$ is the only possible arc. Note that $T_i$ also contains the two arcs connecting $p$ to $u$ and to $w$ because any arc that crosses one of the two must also cross $\beta$ and so cannot belong to $T_i$. 

Therefore, when $T_i$ is a triangulation of $P^\mstar$ that contains $\beta$, we will define the contraction of $T_i$ along $\alpha$ as the operation that first performs the flip that replaces $\beta$ by $\alpha$ within $T_i$, and then contracts the resulting triangulation along $\alpha$.

By construction, the triangulations $T'_{i-1}$ and $T'_i$ obtained by contracting $T_{i-1}$ and $T_i$ along $\alpha$ are still either identical or related by a flip. In fact, when $p$ does not lie in $\beta$, $T'_{i-1}$ and $T'_i$ are identical exactly when the flip that transforms $T_{i-1}$ into $T_i$ affects a triangle incident to arc $\alpha$, as in the proof of Theorem \ref{Ptheorem.6.2}. When $p$ lies in $\beta$, this is still true, except if $T_{i-1}$ (resp. $T_i$) contains $\beta$ and the flip that transforms $T_{i-1}$ into $T_i$ introduces (resp. removes) the puncture. It will be important to remark that, if $T_i$ is obtained from $T_{i-1}$ by a flip that introduces the puncture, and neither triangulations contain $\beta$, then $T'_{i-1}$ and $T'_i$ are identical because this flip must introduce $\alpha$, and because we are not in the exception just mentioned.

As $k$ is the least number of flips required to transform $T_0$ into $T_k$, the sequence of triangulations $T'_0$, ..., $T'_k$ still provides a geodesic between $T_0$ and $T_k$. Hence $T'_{i-1}$ and $T'_i$ are never identical. Assume that some triangulation among $T_0$, ..., $T_k$ is a triangulation of $P^\mstar$. Let $j$ be the least index so that $T_j$ is a triangulation of $P^\mstar$ and $l$ be the index so that $T_i$ is a triangulation of $P^\mstar$ when $j\leq{i}<l$ and $T_l$ is a triangulation of $P$. As $p$ does not belong to the interior of the triangle with vertices $u$, $v$, and $w$, contracting a triangulation of $P^\mstar$ along $\alpha$ results in a triangulation that cannot contain $\beta$. This is the case for triangulations $T'_j$ to $T'_{l-1}$. As a consequence, $T_{j-1}$ and $T_l$ cannot both contain $\beta$. Otherwise, by Lemma 3 from \cite{SleatorTarjanThurston1988}, $T'_j$, ..., $T'_{l-1}$ would contain $\beta$. Assume without loss of generality that this flip is the one that introduces the puncture within $T_{j-1}$. In this case, neither $T_{j-1}$, nor $T_j$ contains $\beta$, and by the above remark, $T'_{j-1}$ is identical to $T'_j$, a contradiction.
\end{proof}

In light of the result above, we suggest the following conjecture. 
\begin{conjecture}\label{ConjPP}
Let $P^\mstar$ be a punctured convex $n$-gon. The natural embedding
$$
\A_n \hookrightarrow {\overline{\mathcal F}}(P^{\mstar})
$$
is strongly convex regardless of the placement of the puncture.
\end{conjecture}

The method we use to prove Theorem \ref{thm:SCP} is highly dependent on the placement of the puncture. In order to prove Conjecture \ref{ConjPP}, it will probably be required to use a different approach which relies less on the local combinatorics. 

In addition to $\A_n$, the other natural embedded subgraph of ${\overline{\mathcal F}}(P^{\mstar})$ is $\mathcal{F}(P^{\mstar})$ and we could ask about the geometry of its embedding. We shall show that the situation is very different by considering the two triangulations of a heptagon shown in Fig. \ref{Pfigure.5.2}. Using these, we'll show that the natural embedding $\mathcal{F}(P^{\mstar})\hookrightarrow {\overline{\mathcal F}}(P^{\mstar})$ is generally not strongly convex.

\begin{theorem}
Consider a convex $n$-gon $P$ where $n\geq7$. One can find a placement of the puncture so that, for some triangulations $U$ and $V$ of the resulting punctured polygon $P^\mstar$, no geodesic between $U$ and $V$ is entirely contained in $\FF(P^\mstar)$.
\end{theorem}

\begin{proof}
We just need to prove that the result holds for the triangulations shown in Fig. \ref{Pfigure.5.2}.
\begin{figure}[b]
\begin{centering}
\includegraphics{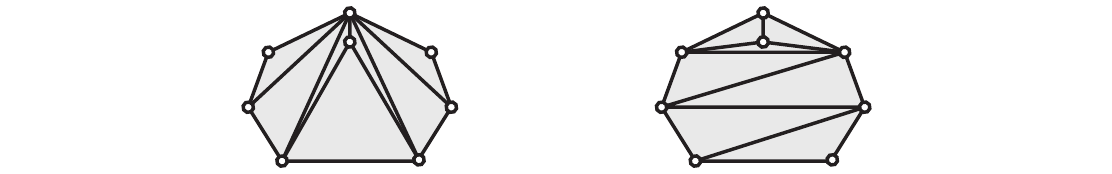}
\caption{Two triangulations of a punctured Euclidean heptagon.}\label{Pfigure.5.2}
\end{centering}
\end{figure}
In order to extend it to triangulations of a larger punctured $n$-gon $P^\mstar$, it suffices to glue the triangulations of Fig. \ref{Pfigure.5.2} along one boundary edge to a fixed triangulation of a $(n-5)$-gon. By Corollary \ref{Pcorollary.5.2}, the geodesics between the two resulting triangulations of $P^\mstar$ will keep the desired property.

First observe that the triangulations shown in Fig. \ref{Pfigure.5.2} differ by $6$ interior arcs. Hence a geodesic between then that never removes the puncture must have length at least $6$. In fact, the length of such a geodesic is at least $7$ because none of the flips that can be performed in either of these triangulations introduces an edge of the other triangulation. Now observe that flipping the puncture out of the triangulation shown on the left of Fig. \ref{Pfigure.5.2} results in a triangulation of the heptagon whose four interior edges are incident to the same vertex. Such a triangulation is distant by at most $4$ flips from any other triangulation of the heptagon. Hence, the two triangulations in Fig. \ref{Pfigure.5.2} are at most $6$ flips away if the removal and the insertion of the puncture are allowed.
\end{proof}

\addcontentsline{toc}{section}{References}
\bibliographystyle{amsplain}
\bibliography{Pointihedra}

{\em Addresses:}\\
Department of Mathematics, University of Fribourg, Switzerland \\
LIPN, Universit{\'e} Paris 13, Villetaneuse, France\\
{\em Emails:} \href{mailto:hugo.parlier@unifr.ch}{hugo.parlier@unifr.ch}, \href{mailto:lionel.pournin@univ-paris13.fr}{lionel.pournin@univ-paris13.fr}\\

\end{document}